\documentclass[letterpaper,11pt,oneside,reqno]{amsart}
\usepackage[T2A]{fontenc}%
\usepackage[utf8]{inputenc}%
\usepackage[english]{babel}%
\usepackage{amsmath,amssymb,amsthm,amsfonts}%
\usepackage{hyperref}%
\usepackage{enumerate}%
\usepackage{array}%
\usepackage{graphicx}

\usepackage{mhequ}
\usepackage[margin=1.1in]{geometry}

\numberwithin{equation}{section}

\newcommand{\h}{\mathcal{H}}
\newcommand{\E}{\hat{\mathbb{E}}}
\newcommand{\W}{\mathbb{W}}
\newcommand{\tr}{\operatorname{Tr}}

\newtheorem{proposition}{Proposition}[section]
\newtheorem{lemma}[proposition]{Lemma}

\newtheorem{theorem}[proposition]{Theorem}
\newtheorem*{theorem*}{Theorem}
\theoremstyle{definition}
\newtheorem{definition}[proposition]{Definition}
\newtheorem{remark}[proposition]{Remark}

\begin{document}

\title[$G$-Gaussian random fields and stochastic quantization]
{$G$-Gaussian random fields and stochastic quantization under nonlinear expectation}
\author[H. Hu]{Haoran Hu}
\address{School of Mathematics, Shandong University, Jinan, China.}
\email{201900100081@mail.sdu.edu.cn}

\date{}
\begin{abstract}
We investigate the application of Parisi-Wu stochastic quantization to the construction of random fields within the sublinear expectation framework. Using the semigroup approach and the infinite dimensional $G$-Ornstein Uhlenbeck process, we derive the unique mild solution to the robust Langevin dynamics of bosonic free field --- a parabolic linear stochastic partial differential equation (SPDE) driven by cylindrical $G$-Brownian motions. Mimicking the linear expectation case, we show the equilibrium distribution of the mild solution is the sublinear expectation analog of the massive Gaussian free field.
\end{abstract}
\maketitle

\setcounter{tocdepth}{1}
\setcounter{tocdepth}{3}



\section{Introduction and main result} 

The theory of nonlinear expectation is a significant probabilistic tool in the study of random models with uncertainty. Indeed, in research areas such as mathematical economics and stochastic finance, the randomness is typically characterised by a family of probability measures $\{\mathbb{P}_{\theta}\}_{\theta\in\Theta}$ and observers are usually not able to determine the true distribution. In 1981, Huber \cite{huber} proposed the upper expectation $\E(X):=\sup_{\theta\in\Theta}\mathbb{E}_{\theta}(X)$ to calculate the robust statistics of such models, and it is obvious that this upper expectation functional $\E$ is nonlinear unless $\Theta$ is a singleton. Similarly, Walley \cite{walley} developed the notion of upper prevision, which is closely related to coherent risk measures developed in latter works by Artzner-Delbaen-Eber-Heath \cite{artzner}, Delbaen \cite{delbaen}, Föllmer and Schied \cite{FS}. From the stochastic analysis perspectives, Peng \cite{Peng97} first discovered that the solution to backward stochastic differential equations (BSDEs) naturally contains the $g$-expectations, which is a special example in the family of $G$-expectations. This family of backward-in-time $g$-expectations has important applications (e.g., see \cite{Peng12}) in the pricing mechanics of financial markets. Motivated by these mathematical financial models with uncertainty and the theory of BSDEs, Peng \cite{Peng07} systematically established the unified theory of nonlinear expectation and the stochastic calculus of $G$-Brownian motions.

In recent works by Ji and Peng \cite{JP18,JP24}, they developed rigorous construction for the $G$-Gaussian random fields, and the solution of stochastic heat equations (SHE) driven by multiplicative $G$-Gaussian noises. By Kolmogorov extension, they established the distribution of a random field by specifying a family of compatible finite dimensional sublinear expectations. Apart from this direct method, one can also construct random fields implicitly using the Parisi-Wu stochastic quantization approach \cite{DH,PW}. In the context of Euclidean quantum field theory (EQFT), this approach suggests that the path integral measure can be realized as the invariant distribution of the corresponding Langevin dynamics, which is described by parabolic stochastic partial differential equations (SPDEs) driven by spacetime Gaussian white noise.

For example, consider a bosonic free field propagating through Euclidean spacetime. It is well-known (see e.g., \cite{LS}) that the massless path integral measure -- Gaussian free field is the unique invariant distribution of the SHE dynamics. More recently, several works \cite{AK,GH,MW,CCHSa,CCHSb} have developed rigorous mathematical approaches to the stochastic quantization of interacting fields and gauge theory in dimensions $d=1,2,3$, under various interesting physical settings. However, these SPDEs are typically singular and non-trivial, requiring technical renormalization (e.g., the regularity structure introduced by Hairer \cite{H14}). In this article, we only investigate the extension of dynamical bosonic free fields to the sublinear expectation case. Unlike \cite{JP24}, where the SHE was solved using martingale measure theory and infinite dimensional stochastic analysis, our main result employs the semigroup approach and demonstrates that the massive $G$-Gaussian free field is the unique equilibrium distribution of the dynamics. Moreover, we expect a general theory of robust stochastic dynamics which serves as a useful tool to analyze a wide class of random fields under nonlinear expectation.

\begin{theorem}\label{main}
    Fix dimension $d\geq1$, variance regime $0<\underline{\sigma}^{2}\leq\bar{\sigma}^{2}$ and mass $m\geq0$. Let $D\subset\mathbb{R}^{d}$ be a bounded domain, and $\{\W(t)\}_{t\ge0}$ be a cylindrical $G$-Brownian motion \eqref{expand} on $L^{2}(D)$ such that for any $n\in\mathbb{N}$, the coefficient $\W_{n}(t)$ is an $N(0,t[\underline{\sigma}^{2},\bar{\sigma}^{2}])$ distributed $G$-Brownian motion (see Section~\ref{pre} for definition). Suppose $\phi(x,t)$, $(x,t)\in D\times\mathbb{R}_{+}$ is the mild solution \eqref{duhamel} to the stochastic reaction-diffusion equation:
    \begin{equation}\label{reac}
        d\phi(x,t)=(\Delta-m^{2})\phi(x,t)dt+d\W(t),\quad \phi(x,0)\in C_{0}^{\infty}(D),\;\phi(x,t)\vert_{x\in \partial D}=0.
    \end{equation}
    Then as $t\to\infty$, $\phi(x,t)$ converges in law (see \eqref{wkconv}) to the massive $G$-Gaussian free field $\Psi(x)$ with Dirichlet boundary condition (refer to Definition~\ref{GGFF}). More precisely, there exists a constant $\alpha(d,D,m)>0$ such that $\Psi$ is a centered $G$-Gaussian stochastic process indexed by Sobolev space $H^{-1}_{0}(D)$, and for any $f,g\in H^{-1}_{0}(D)$, we have
    \begin{equation}\label{1}
        \E\big([\Psi,f][\Psi,g])\leq\frac{\bar{\sigma}^{2}}{2\alpha}\iint_{D^{2}}G_{m}(x,y)f(x)g(y)dxdy,
    \end{equation}
    \begin{equation}\label{2}
        -\E\big(-[\Psi,f][\Psi,g])\geq\frac{\underline{\sigma}^{2}}{2\alpha}\iint_{D^{2}}G_{m}(x,y)f(x)g(y)dxdy.
    \end{equation}
    In the above covariance bounds, $[\cdot,\cdot]$ is an $L^{2}$ bilinear form \eqref{the[]} and $G_{m}$ is the Green's function of the differential operator $(2\pi^{-1})(-\Delta+m^{2})$.
\end{theorem}

In the linear expectation case, the $2$-dimensional massless Gaussian free field (GFF) is a canonical random generalized function $\Psi(z)$, which can be understood as a 2-dimensional time analogue of the Brownian motion. It also arises from many random interface models and exhibits strong conformal symmetry, see \cite{sheffield,WP}. Moreover, for planar domains $D\subset\mathbb{R}^{2}$, the family of random measures $\exp[\gamma \Psi(z)]dz$ with $\gamma\in[0,\sqrt{2})$ provides the probabilistic interpretation \cite{BP} of the Liouville quantum gravity (LQG). From the stochastic quantization perspective, Garban \cite{garban} followed the work of David-Kupiainen-Rhodes-Vargas \cite{DKRV} and constructed the dynamics for Liouville conformal field theory (LCFT) on the sphere $\mathbb{S}^{2}$ and on the torus $\mathbb{T}^{2}$ respectively. The stochastic quantization for general EQFT is also understood as the the continuum scaling limit of the Markovian dynamics for discrete lattice spin systems. For the Glauber and Kawasaki dynamics of lattice models, another interesting research topic is to quantify the convergence speed to the equilibrium distribution. In particular, the spectral gap estimates and the log-Sobolev inequality were shown to be crucial in studying such dynamical behaviors, see \cite{GZ, martinelli} for further introduction.

The following context will be divided into two parts. In Section~\ref{white}, we will review the theory of $G$-Gaussian random fields and explain the stochastic analysis for cylindrical $G$-Brownian motions. Section~\ref{sq} surveys the rigorous structure of the robust Langevin dynamics of massive bosonic free field, meanwhile completing the proof of our main result. Last but not least, the author wish to thank Prof. Shige Peng for his support and many enlightening discussions. This paper was written while the author was an undergraduate student at Shandong University.

\section{Analysis of $G$-Gaussian white noise}
\label{white}

\subsection{Preliminaries on nonlinear expectation}
\label{pre}
In this subsection, we briefly recall the rigorous setting of nonlinear expectation and list a few important theorems that will be needed in the following sections. For detailed explanation, see \cite{JP18,Peng19,Peng05}.

A nonlinear expectation space refers to a triple $(\hat{\Omega},\hat{\h},\E)$, where $\hat{\Omega}$ is a set denoting the sample space, $\hat{\h}$ is a real vector space of functions (i.e., random variables) on $\hat{\Omega}$ such that $c\in\hat{\h}$ and $\vert X\vert\in\hat{\h}$ for all $c\in\mathbb{R}$ and $X\in\hat{\h}$. The third ingredient $\E$ is a nonlinear functional from $\hat{\h}$ to $\mathbb{R}$ which satisfies two properties: (i) Monotonicity: For any $X,Y\in\hat{\h}$ such that $X\geq Y$, we have $\E(X)\geq\E(Y)$; (ii) Constant preserving: $\E(c)=c$. If $\E$ satisfies two additional properties: (iii) Subadditivity: $\E(X+Y)\leq\E(X)+\E(Y)$, $\forall X,Y\in\hat{\h}$ and (iv) Positive homogeneity: $\E(\lambda X)=\lambda\E(X)$ for each $\lambda\geq0$, it is then called a sublinear expectation. In our context, we are mostly interested in sublinear expectations $\E$ with an additional regularity: for any sequence of non-increasing random variables $\{X_{n}\}_{n=1}^{\infty}\subset\hat{\h}$, such that $X_{n}(\omega)\downarrow0$ for all $\omega\in\hat{\Omega}$, we get $\E(X_{n})\downarrow0$ as $n\to\infty$. One easily finds that measure theoretic probability is just a special case with linear expectation.

Similar to the idea of weak topology on the space of all probability measures, we can study the distribution of a random variable $X\in\hat{\h}$ (or a random vector $(X_{1},...,X_{d})\subset\hat{\h}$) by testing it against nice functions. Let $C_{l,Lip}(\mathbb{R}^{d})$ denote the real linear space of functions $\varphi$ with regularity:
\begin{equation}\label{lipschitz}
    \vert\varphi(x)-\varphi(y)\vert\leq C(1+\vert x\vert^{m}+\vert y\vert^{m})\vert x-y\vert,\quad\forall x,y\in\mathbb{R}^{d}.
\end{equation}
The constants $C>0$ and $m\in\mathbb{N}$ are independent of $x$ and $y$ but dependent on $\varphi$. For a given $d$-dimensional random vector $X=(X_{1},...,X_{d})$, the distribution is defined to be the triple $(\mathbb{R}^{d},C_{l,Lip}(\mathbb{R}^{d}),\mathbb{F}_{X})$ such that $\mathbb{F}_{X}(\varphi):=\E(\varphi(X))$ for all $\varphi\in C_{l,Lip}(\mathbb{R}^{d})$. According to this identification, we immediately see that the distribution is nothing but a new well-defined sublinear expectation space with Lipschitz random variables, which are much easier to work with. Given two random vectors of the same dimension $X,Y$, they are called identically distributed iff $\mathbb{F}_{X}=\mathbb{F}_{Y}$. Furthermore, suppose $X$ is $m$-dimensional and $Y$ is $n$-dimensional, $X,Y$ are said to be independent iff for any $\varphi\in C_{l,Lip}(\mathbb{R}^{m+n})$, we have
\begin{equation}
\label{independent}
    \mathbb{F}_{(X,Y)}(\varphi)=\E(\E(\varphi(x,Y))_{x=X}).
\end{equation}
Following this idea, we can construct i.i.d copies of sublinear expectations by defining a product space $(\hat{\Omega}^{\otimes n},\hat{\h}^{\otimes n},\E^{\otimes n})$. In particular, the space $\hat{\Omega}^{\otimes n}$ is nothing but the Cartesian product. The set of all random variables $\hat{\h}^{\otimes n}$ on this product space is given by
\begin{equation*}
   \hat{\h}^{\otimes n}:=\{\varphi(X_{1},...,X_{n});\forall \varphi\in C_{l,Lip}(\mathbb{R}^{n}),X_{i}\in\hat{\h}_{i},i=1,...,n\}.
\end{equation*}
The new multi-dimensional sublinear expectation $\E^{\otimes n}$ is defined inductively by 
\begin{equation*}
    \E^{\otimes n}(\varphi(X_{1},...,X_{n}))=\E(\E^{\otimes n-1}(\varphi(x_{1},X_{2}...,X_{n}))_{x_{1}=X_{1}}).
\end{equation*}
One can extend such definition to an infinite product. It was shown that an analogues central limit theorems also applies to certain sequences of i.i.d random vectors with sublinear expectation. The universality classes are characterized by the $G$-maximal distributions and the $G$-Gaussian distributions see Definition~\ref{Ggaussian} and \cite[Chapter~2]{Peng19} for further details. 
Finally, a sequence of random variables $\{X_{n}\}_{n=1}^{\infty}$ is said to converge in law to $X$ as $n\to\infty$ iff 
\begin{equation}
\label{wkconv}
\lim_{n\to\infty}\mathbb{F}_{X_{n}}(\varphi)=\mathbb{F}_{X}(\varphi)
\end{equation} 
holds for any bounded Lipschitz test function $\varphi$.

In fact, a sublinear expectation can be encoded by a family of probability measures on $(\hat{\Omega},\sigma(\hat{\h}))$. This follows easily from an application of the Daniell-Stone theorem (e.g., see \cite{cohn}) and one can henceforth develop the idea of capacity and quasi-sure analysis for random variables. Once again, we see that if the set of uncertain probability measures is just a singleton, the entire structure reduces to linear expectation.
\begin{theorem}[Robust Daniell-Stone]
\label{RDS}
    Suppose $\E$ is a sublinear expectation on $(\hat{\Omega},\hat{\h})$ which satisfies the regularity property, then there exists a weakly compact set of probability measures $\Theta$ on the $\sigma$-algebra $\sigma(\hat{\h})=\sigma(\{X^{-1}(A);X\in\hat{\h},A\in\mathcal{B}(\mathbb{R})\})$ such that
    \begin{equation*}
        \E(X)=\max_{\mathbb{P}\in\Theta}\int_{\Omega}X(\omega)d\mathbb{P},\quad\forall X\in\hat{\h}.
    \end{equation*}
    In addition, if we define a function $c:\sigma(\hat{\h})\to\mathbb{R}_{+}$ by 
    \begin{equation*}
        c(A)=\sup_{\mathbb{P}\in\Theta}\mathbb{P}(A),\quad\forall A\in\sigma(\hat{\h}),
    \end{equation*}
    then $c$ is a Choquet capacity and any two random variables $X$ and $Y$ are said to agree qausi-surely (denoted by q.s.) iff $c(X\neq Y)=0$.
\end{theorem}

The theory on stochastic processes and random fields (i.e., a family of random vectors $X_{\gamma}$ indexed by a parameter $\gamma\in\Gamma$) with nonlinear expectation will be the building block for stochastic analysis with uncertainty. To study Wiener process and white noise, we need to introduce the sublinear expectation analog of the $d$-dimensional centered Gaussian distribution. 
\begin{definition}
\label{Ggaussian}
    A $d$-dimensional random vector $X=(X_{1},...,X_{d})$ is said to be centered $G$-Gaussian distributed if
    \begin{equation*}
        aX+b\bar{X}\overset{d}{=}\sqrt{a^{2}+b^{2}}X,\quad\forall a,b\geq0.
    \end{equation*}
    where $\bar{X}$ is an independent copy of X.
\end{definition}

This definition is somewhat misleading because it doesn't manifest the idea of $G$. Indeed, the most important characterization of the distribution $\mu_{X}$ of a centered Gaussian random vector $X$ is the covariance matrix, which is equivalent to the functional $G(Q):=\mathbb{E}(X^{T}QX)$ on the space of all $d\times d$ real symmetric matrices $Q\in\mathbb{S}(d)$. In our case, we would expect $G$ to be sublinear and to uniquely characterize the desired distribution. Moreover, the distribution of $G$-Gaussian random vector is strongly related to the viscosity solution of the nonlinear heat equation: $\partial_{t}u-G(\nabla^{2}u)=0$. This leads to many interesting consequences. For example, the idea of viscosity solutions to nonlinear PDEs applies to the study of HJB equations, which connects to various results in stochastic control problems, for example, see \cite{YZ}. We summarize the above heuristics of $G$-Gaussian distribution in the following theorem.

\begin{theorem}
    A $d$-dimensional centered $G$-Gaussian distribution is uniquely characterized by the covariance functional $G=G_{X}:\mathbb{S}(d)\to\mathbb{R}$ given by
    \begin{equation}
    \label{cov}
        G_{X}(Q):=\frac12\E(\langle QX,X\rangle),\quad Q\in\mathbb{S}(d).
    \end{equation}
    Two centered $G$-Gaussian random vectors are identically distributed iff their covariance functionals coincide. Conversely, for any given sublinear functional $G:\mathbb{S}(d)\to\mathbb{R}$ with monotonicity, constant preserving, subadditivity and positive homogeneity properties, there exists a unique (up to identical distributions) $d$-dimensional centered $G$-Gaussian random vector $X$ satisfying \eqref{cov}. 
\end{theorem}

As a direct corollary of Theorem~\ref{RDS}, the covariance functional \eqref{cov} is alternatively given in a variational formula:
\begin{equation}
    G_{X}(Q)=\frac12 \sup_{S\in\Theta}\tr(QS),
\end{equation}
where $\Theta$ is a closed bounded convex subset of $\mathbb{S}(d)$. In the $1$-dimensional case, it is a compact positive interval which is strictly bounded away from $0$ and infinity, i.e., $\Theta=[\underline{\sigma}^{2},\bar{\sigma}^{2}]$, $0<\underline{\sigma}\leq\bar{\sigma}$, and the corresponding nonlinear heat equation is known as the Barenblatt equation. Simple calculations yields 
\begin{equation}
\label{covariance-bd}
\E(X^{2})=\bar{\sigma}^{2},\quad-\E(-X^{2})=\underline{\sigma}^{2}.
\end{equation}
Hence the interval $\Theta$ can be naturally understood as the variance regime of the centered random variable: $X\sim N(0,[\underline{\sigma}^{2},\bar{\sigma}^{2}])$. We are only interested in the centered $G$-Gaussian distributions, thus `centered' will be omitted in the following contexts for brevity. At $d=1$, we can compute the $G$-Gaussian distributions by proving the covariance bounds \eqref{covariance-bd}.

A random field $(X_{\gamma})_{\gamma\in\Gamma}$ is called $G$-Gaussian distributed if for any $\underline{\gamma}=(\gamma_{1},...,\gamma_{n})\subset\Gamma$ and $n\geq1$, the finite-dimensional distribution of $(X_{\gamma_{1}},...,X_{\gamma_{n}})$ is $n$-dimensional $G$-Gaussian. If the parameter set is the temporal regime $\Gamma=[0,\infty)$, there exists a 1-dimensional canonical continuous Lévy process --- $G$-Brownian motion  $(B_{t})_{t\ge0}$, which is defined to satisfy: (i) $B_{0}=0$; (ii) For any $t\ge s\ge0$, the increment $B_{t}-B_{s}$ is $N(0,(t-s)[\underline{\sigma}^{2},\bar{\sigma}^{2}])$ distributed and is independent of the past $(B_{r})_{0\le r\le s}$. This definition is consistent with the classical ones in the sense that Itô's stochastic analysis with uncertainty can be similarly developed, see \cite{Peng19}. However, by Definition~\ref{Ggaussian}, the $G$-Brownian motion is not a $G$-Gaussian process. We postpone the discussion of this non-$G$-Gaussian anomaly in Section~\ref{sec:white}.

In this paper, our study on stochastic PDEs driven by $G$-Gaussian random noise is based on the following key result, which is essentially a Kolmogorov extension theorem that holds for arbitrary random fields with nonlinear expectation.
\begin{theorem}[\cite{JP18}]
\label{compactibility}
    Let $\mathcal{J}_{\Gamma}$ denote the family of all finite indices of the form $\underline{\gamma}=(\gamma_{1},...,\gamma_{n})\subset\Gamma$, and suppose $\{\mathbb{F}_{\underline{\gamma}},\underline{\gamma}\in\mathcal{J}_{\Gamma}\}$ is a family of finite-dimensional distributions satisfying the following conditions:
    \begin{enumerate}[(i)]
        \item $\mathbf{Compatibility}$: For any $(\gamma_{1},...,\gamma_{n},\gamma_{n+1})\in\mathcal{J}_{\Gamma}$ and $\varphi\in C_{l,Lip}(\mathbb{R}^{n})$, we have
        \begin{equation*}
            \mathbb{F}_{(\gamma_{1},...,\gamma_{n})}(\varphi)=\mathbb{F}_{(\gamma_{1},...,\gamma_{n+1})}(\tilde{\varphi})
        \end{equation*}
        where $\tilde{\varphi}(x_{1},...,x_{n},x_{n+1})=\varphi(x_{1},...,x_{n})$.
        \item $\mathbf{Symmetry}$: For any permutation $\pi$ of $\{1,...,n\}$, we have
        \begin{equation*}
            \mathbb{F}_{(\gamma_{\pi(1)},...,\gamma_{\pi(n)})}(\varphi)=\mathbb{F}_{(\gamma_{1},...,\gamma_{n})}(\varphi)
        \end{equation*}
        for any $\varphi\in C_{l,Lip}(\mathbb{R}^{n})$ and $(\gamma_{1},...,\gamma_{n})\in\mathcal{J}_{\Gamma}$.
    \end{enumerate}
    Then there exists a sublinear expectation space $(\hat{\Omega},\hat{\h},\E)$ and a random field $(X_{\gamma})_{\gamma\in\Gamma}$ defined on this space such that the finite-dimensional distribution of $(X_{\gamma})_{\gamma\in\Gamma}$ coincide with $\{\mathbb{F}_{\underline{\gamma}},\underline{\gamma}\in\mathcal{J}_{\Gamma}\}$. Moreover, such random fields are uniquely determined in law.
\end{theorem}
According to this result, to construct a $G$-Gaussian field, we only need to specify a compatible family of sublinear covariance functionals $\{G_{\underline{\gamma}};\underline{\gamma}\in\mathcal{J}_{\Gamma}\}$.
However, due to nonlinearity of the expectation, we still lose some useful characterizations such as the Cameron-Martin spaces of Guassian measures. Thus in the sublinear expectation case, several additional modifications are needed. For simplicity of notations in the following contexts, we sometimes write $X\lesssim Y$ (or $X\gtrsim Y$) as shorthand for the inequality $X\leq cY$ ($X\geq cY$) for some constant $c>0$.

\subsection{$G$-Gaussian white noise}
\label{sec:white}

From now on, we always assume a priori that $(\mathbf{H},\langle\cdot,\cdot\rangle)$ is a real separable Hilbert space. Under the framework of probability theory \cite{AO,hairer,oksendal}, the Gaussian white noise is a linear isometry from $\mathbf{H}$ to the Gaussian subspace of $L^{2}(\Omega)$. However, under sublinear expectation, it is impossible to construct a $G$-Gaussian white noise $(\mathbb{W}_{h})_{h\in\mathbf{H}}$ such that $\mathbb{W}_{h}$ and $\mathbb{W}_{k}$ are independent whenever $h\perp k$. Indeed, if $W_{1}$ and $W_{2}$ are i.i.d. $N(0,[\underline{\sigma}^{2},\bar{\sigma}^{2}])$ distributed, the couple $(W_{1},W_{2})$ is not even a $2$-dimensional $G$-Gaussian vector (see \cite[Excercise~2.5.1]{Peng19}). Hence we can only define it through covariance functions and apply Kolmogorov extension. Assume $\underline{h}:=(h_{1}, ..., h_{n})\in\mathcal{J}_{\mathbf{H}}$, define a matrix subset by 
\begin{equation}\label{matrix}
\Theta=\begin{pmatrix}\langle h_{1}, h_{1}\rangle\theta&\langle h_{1}, h_{2}\rangle\theta&\cdot\cdot\cdot&\langle h_{1}, h_{n}\rangle\theta\\
\vdots&\vdots&\vdots&\vdots\\
\langle h_{n}, h_{1}\rangle\theta&\langle h_{n}, h_{2}\rangle\theta&\cdot\cdot\cdot&\langle h_{n}, h_{n}\rangle\theta\\ \end{pmatrix}\subset\mathbb{S}(n),
\end{equation}
where $\theta$ is a coefficient satisfying $0<\underline{\sigma}^{2}\le\theta\le\bar{\sigma}^{2}$. Since $\Theta$ is a continuous injection from $[\underline{\sigma}^{2}, \bar{\sigma}^{2}]$ to $\mathbb{S}(n)$, it's naturally a closed bounded convex subset of $\mathbb{S}(n)$ and the function
\begin{equation}\label{G}
G_{\underline{h}}(A):=\frac{1}{2}\displaystyle\sup_{\theta\in[\underline{\sigma}^{2}, \bar{\sigma}^{2}]}\tr(AB),\quad B\in\Theta
\end{equation}
is a well-defined sublinear functional on $\mathbb{S}(n)$ and therefore we have a $G_{\underline{h}}$-Gaussian distributed random vector $(\W_{h_{1}},...,\W_{h_{n}})$. The following lemma and Theorem~\ref{compactibility} together imply that there exists a $G$-Gaussian random field $\W=(\W_{h})_{h\in\mathbf{H}}$ with finite dimensional distribution characterised by the family $(G_{\underline{h}})_{\underline{h}\in\mathcal{J}_{\mathbf{H}}}$. The proof is nothing but direct verification, and the resulting random field $\W$ is called a $G$-white noise on $\mathbf{H}$.

\begin{lemma}
The family of functions $\{G_{\underline{h}}\vert\;\forall \underline{h}=(h_{1}, ..., h_{n})\in\mathbf{H}\}$ exhibits compatibility in the following sense:
\begin{enumerate}[(i)]
    \item  For any $A\in\mathbb{S}(n)$, $\underline{h}\in\mathcal{J}_{\mathbf{H}}$ and $h_{n+1}\in\mathbf{H}$, we have $G_{h_{1}, ..., h_{n+1}}\begin{pmatrix}A&0\\0&0\\ \end{pmatrix}=G_{h_{1}, ..., h_{n}}(A)$.

 \item For any permutation $\sigma$ of $n$ elements and any $A=(a_{ij})\in\mathbb{S}(n)$, 
$$
G_{h_{\sigma(1)}, ..., h_{\sigma(n)}}(a_{ij})=G_{h_{1}, ..., h_{n}}(a_{\sigma^{-1}(i)\sigma^{-1}(j)}).
$$
\end{enumerate}
\end{lemma}

To introduce our first result, we need an analogous notion of the $L^{p}$ spaces. Fix any $p\ge1$, it is apparent that $\hat{\h}_{0}^{p}:=\{X\in\hat{\h};\E(\vert X\vert^{p})=0\}$ is a linear subspace and we are able to introduce the quotient $\hat{\h}^{p}:=\hat{\h}/\hat{\h}_{0}^{p}$. By setting $\Vert\hat{X}\Vert_{p}:=\E(\vert X\vert^{p})^{1/p}$ for any equivalence class $\hat{X}$, we obtain a well-defined norm $\Vert\cdot\Vert_{p}$, and the completion of $\hat{\h}^{p}$ under this norm is a Banach space $\tilde{\h}^{p}$. Even though this completion is mostly nontrivial, one can still continuously extend the domain of the sublinear expectation functional $\E$ from $\hat{\h}$ to all of $\tilde{\h}^{p}$. 

\begin{theorem}\label{dist}
Let $\W$ be the $G$-white noise on $\mathbf{H}$ with distribution specified by \eqref{matrix} and \eqref{G}, then for any two vectors $h,k\in\mathbf{H}$ with $\langle h, k\rangle\ge0$, one has
\begin{equation}\label{11}
        \W_{h}\backsim N(0, \Vert h\Vert^{2}[\underline{\sigma}^{2}, \bar{\sigma}^{2}]),\quad\hat{\mathbb{E}}(\mathbb{W}_{h}\mathbb{W}_{k})=\langle h, k\rangle\bar{\sigma}^{2},\quad -\hat{\mathbb{E}}(-\mathbb{W}_{h}\mathbb{W}_{k})=\langle h, k\rangle\underline{\sigma}^{2}.
\end{equation}
Moreover, suppose $\{e_{i},i\in\mathbb{N}_{+}\}$ is an orthonormal basis of $\mathbf{H}$, and $h=\sum_{i\ge1}h_{i}e_{i}:=\sum_{i\ge1}\langle h,e_{i}\rangle e_{i}$, we have the distributional identity 
\begin{equation}\label{22}
\sum_{i=1}^{\infty}h_{i}\mathbb{W}_{e_{i}}\overset{d}{=}\mathbb{W}_{h},
\end{equation}
where the series on the left converges in Banach spaces $(\tilde{\mathcal{H}}^{n}, \lVert \cdot\rVert_{n})$ for any $ n\in\mathbb{N}_{+}$. 

\end{theorem}

\begin{remark}
This result is a very basic characterization of the distribution of $G$-white noise, and as we can see, due to the uncertainty of variance, we only have preservation of inner product \eqref{11} in a weaker sense. However, the linearity of $h\mapsto \W_{h}$ and Fourier expansion \eqref{22} still holds in the common $L^{2}$ sense, which means, in order to construct the distribution of an arbitrary $\W_{h}$, one only need to specify a sequence of standard $G$-Gaussians $\{\W_{e_{i}}\}_{i\ge1}$. 
\end{remark}

\begin{proof}[Proof of Theorem~\ref{dist}]
The distribution of $\W_{h}$ is immediate from the definition of $G_{h}$: 
\begin{equation*}
G_{h}(a)=\frac{1}{2}\displaystyle\sup_{\theta\in[\underline{\sigma}^{2}, \bar{\sigma}^{2}]}a\Vert h\Vert^{2}\theta=\frac{1}{2}\big[(a\vee0)\Vert h\Vert^{2}\bar{\sigma}^{2}-(a\wedge0)\Vert h\Vert^{2}\underline{\sigma}^{2}\big],\quad\forall a\in\mathbb{R}.
\end{equation*}
For the covariance, take $A=\begin{pmatrix}0&1\\1&0\\ \end{pmatrix}$, we have two identities
\begin{align*}
G_{h,k}(A)
&=\frac{1}{2}\displaystyle\sup_{\theta\in[\underline{\sigma}^{2}, \bar{\sigma}^{2}]}\tr\begin{pmatrix}0&1\\1&0\\ \end{pmatrix}\begin{pmatrix}\Vert h\Vert^{2}\theta&\langle h,k\rangle\theta\\\langle h,k\rangle\theta&\Vert k\Vert^{2}\theta\\ \end{pmatrix}=\langle h, k\rangle\bar{\sigma}^{2},\\
&G_{h,k}(A)=\frac{1}{2}\hat{\mathbb{E}}(A\mathbb{W}_{h,k},\mathbb{W}_{h,k})=\hat{\mathbb{E}}(\mathbb{W}_{h}\mathbb{W}_{k}).
\end{align*}
In the above, $\mathbb{W}_{h,k}$ denotes the $2d$ $G$-Gaussian vector $(\mathbb{W}_{h},\mathbb{W}_{k})$. Similarly, $-\hat{\mathbb{E}}(-\mathbb{W}_{h}\mathbb{W}_{k})$ is calculated by substituting $-A$ for $A$.

For the proof of \eqref{22}, note that any finitely truncated expansion $\sum_{i=1}^{n}h_{i}\mathbb{W}_{e_{i}}$ is a $1d$ $G$-Gaussian distributed variable. It suffices to show the sequence $\sum^{n}_{i=1}h_{i}\mathbb{W}_{e_{i}},n\in\mathbb{N}_{+}$ is in fact Cauchy under all $L^{n}$ norms. The idea is to calculate the lower bounds and upper bounds for the uncertain absolute moments:
\begin{align*}
\lim\limits_{n,m\to\infty}\hat{\mathbb{E}}\Big\vert\sum_{i=n+1}^{m}h_{i}\mathbb{W}_{e_{i}}\Big\vert^{k}=\Bigg\{
           \begin{array}{cc}
               \lim\limits_{n,m\to\infty}\frac{2(k-1)!!\Vert\sum_{i=n+1}^{m}h_{i}e_{i}\Vert^{k}\bar{\sigma}^k}{\sqrt{2\pi}}, & \mathrm{k\:is\:odd}. \\
              \lim\limits_{n,m\to\infty}(k-1)!!\Vert\sum_{i=n+1}^{m}h_{i}e_{i}\Vert^{k}\bar{\sigma}^k,  & \mathrm{k\:is\:even}. 
           \end{array} 
\end{align*}
\begin{align*}
\lim\limits_{n,m\to\infty}-\hat{\mathbb{E}}\Big(-\Big\vert\sum_{i=n+1}^{m}h_{i}\mathbb{W}_{e_{i}}\Big\vert^{k}\Big)=\Bigg\{
           \begin{array}{cc}
               \lim\limits_{n,m\to\infty}\frac{2(k-1)!!\Vert\sum_{i=n+1}^{m}h_{i}e_{i}\Vert^{k}\underline{\sigma}^k}{\sqrt{2\pi}}, & \mathrm{k\:is\:odd}. \\
              \lim\limits_{n,m\to\infty}(k-1)!!\Vert\sum_{i=n+1}^{m}h_{i}e_{i}\Vert^{k}\underline{\sigma}^k,  & \mathrm{k\:is\:even}. 
           \end{array} 
\end{align*}
Indeed, this follows directly from an easy application of Robust Daniell-Stone (i.e., Theorem~\ref{RDS}). Since $\sum_{i\ge1}h_{i}e_{i}$ converges in the Hilbert space, we have 
$$
\lim_{n,m\to\infty}\hat{\mathbb{E}}\Big\vert\sum_{i=n+1}^{m}h_{i}\mathbb{W}_{e_{i}}\Big\vert^{k}=\lim_{n,m\to\infty}-\hat{\mathbb{E}}\Big(-\Big\vert\sum_{i=n+1}^{m}h_{i}\mathbb{W}_{e_{i}}\Big\vert^{k}\Big)=0.
$$
By straightforward modification, one obtains the Lyapunov's inequality for sublinear expectation, thus the limit random variables in $(\tilde{\mathcal{H}}^{n})_{n\ge1}$ must be identical. Let's denote this common limit by $\mathbb{X}$ and prove that it is identically distributed with $\mathbb{W}_{h}$. For any positive constant $a>0$, by convergence
$$
\frac{1}{2}\hat{\mathbb{E}}(a\mathbb{X}^{2})=\lim_{n\to\infty}\frac{1}{2}\hat{\mathbb{E}}\Big(a\big[\sum_{i=1}^{n}h_{i}\mathbb{W}_{e_{i}}\big]^{2}\Big)=\lim_{n\to\infty}\frac{1}{2}\Big((a\vee0)\sum_{i=1}^{n}h_{i}^{2}\bar{\sigma}^{2}-(a\wedge0)\sum_{i=1}^{n}h_{i}^{2}\underline{\sigma}^{2}\Big)=G_{h}(a).
$$
We are finally left to show the $G_{h}$-Gaussianity of $\mathbb{X}$. Assume $(\overline{\mathbb{W}}_{e_{i}})_{i\ge1}$ and $\overline{\mathbb{X}}$ are i.i.d. copies of $(\mathbb{W}_{e_{i}})_{i\ge1}$ and $\mathbb{X}$ respectively, we fix an arbitrary $\varphi\in C_{l,Lip}(\mathbb{R})$. By Lipschitz regularity \eqref{lipschitz}, there exits a positive exponent $m>0$ such that
\begin{align}\label{eq:lip}
\hat{\mathbb{E}}&\Big\vert\varphi\Big(\sum^{n}_{i=1}h_{i}\mathbb{W}_{e_{i}}+\sum^{n}_{i=1}h_{i}\overline{\mathbb{W}}_{e_{i}}\Big)-\varphi\Big(\mathbb{X}+\overline{\mathbb{X}}\Big)\Big\vert \notag\\
&\lesssim\hat{\mathbb{E}}\Big(1+\Big\vert\displaystyle\sum^{n}_{i=1}h_{i}\mathbb{W}_{e_{i}}+\sum^{n}_{i=1}h_{i}\overline{\mathbb{W}}_{e_{i}}\Big\vert^{m}+\Big\vert\mathbb{X}+\overline{\mathbb{X}}\Big\vert^{m}\Big)
  \Big(\Big\vert\displaystyle\sum^{n}_{i=1}h_{i}\mathbb{W}_{e_{i}}-\mathbb{X}\vert+\vert\sum^{n}_{i=1}h_{i}\overline{\mathbb{W}}_{e_{i}}-\overline{\mathbb{X}}\Big\vert\Big) \notag\\
  &\lesssim\Bigg[\hat{\mathbb{E}}\Big(1+\Big\vert\displaystyle\sum^{n}_{i=1}h_{i}\mathbb{W}_{e_{i}}+\displaystyle\sum^{n}_{i=1}h_{i}\overline{\mathbb{W}}_{e_{i}}\Big\vert^{m}+\Big\vert\mathbb{X}+\overline{\mathbb{X}}\Big\vert^{m}\Big)^{2}\hat{\mathbb{E}}
  \Big(\Big\vert\displaystyle\sum^{n}_{i=1}h_{i}\mathbb{W}_{e_{i}}-\mathbb{X}\Big\vert+\Big\vert\sum^{n}_{i=1}h_{i}\overline{\mathbb{W}}_{e_{i}}-\overline{\mathbb{X}}\Big\vert\Big)^{2}\Bigg]^{1/2}.
\end{align}
Once again, using the $L^{n}$ converges property proved above, the first sublinear expectation in \eqref{eq:lip} is bounded, meanwhile the second one converges to 0 as $n\to\infty$. Thus we have
\begin{equation}\label{33}
\hat{\mathbb{E}}\Big\vert\varphi\Big(\sum^{n}_{i=1}h_{i}\mathbb{W}_{e_{i}}+h_{i}\overline{\mathbb{W}}_{e_{i}}\Big)-\varphi(\mathbb{X}+\overline{\mathbb{X}})\Big\vert\longrightarrow0,\quad n\to\infty.
\end{equation}
Similarly, 
\begin{align}\label{44}
\hat{\mathbb{E}}&\Big\vert\varphi\Big(\sqrt{2}\sum^{n}_{i=1}h_{i}\mathbb{W}_{e_{i}}\Big)-\varphi(\sqrt{2}\mathbb{X})\Big\vert\notag\\
&\lesssim\hat{\mathbb{E}}\Big(1+\Big\vert\sum^{n}_{i=1}h_{i}\mathbb{W}_{e_{i}}\Big\vert^{m}+\vert\mathbb{X}\vert^{m}\Big)\Big(\Big\vert\sum^{n}_{i=1}h_{i}\mathbb{W}_{e_{i}}-\mathbb{X}\Big\vert\Big)\longrightarrow 0,\quad n\to\infty.
\end{align}
According to \cite[Exercise~2.5.9]{Peng19}, \eqref{33} and \eqref{44} together implies
\begin{equation*}
\hat{\mathbb{E}}\varphi(\mathbb{X}+\overline{\mathbb{X}})
=\lim\limits_{n\to\infty}\hat{\mathbb{E}}\varphi\Big(\sum^{n}_{i=1}h_{i}\mathbb{W}_{e_{i}}+\sum^{n}_{i=1}h_{i}\overline{\mathbb{W}}_{e_{i}}\Big)
=\lim\limits_{n\to\infty}\hat{\mathbb{E}}\varphi\Big(\sqrt{2}\sum^{n}_{i=1}h_{i}\mathbb{W}_{e_{i}}\Big)=\hat{\mathbb{E}}\varphi(\sqrt{2}\mathbb{X}).
\end{equation*}
This concludes that $\mathbb{X}$ is indeed $G_{h}$-Gaussian distributed.
\end{proof}

\begin{remark}\label{moment-method}
    The proof of Theorem~\ref{dist} shows a moment method for $G$-Gaussian random fields. One can determine the invariant distributions of a $G$-Gaussian dynamics by calculating the large time asymptotics of the absolute moments.
\end{remark}

\subsection{Cylindrical $G$-Brownian motion and stochastic integration}

Before we start constructing our spacetime random noise, a keen observer might notice that if we let $\mathbf{H}:=L^{2}(D\times[0,\infty))$, where $D\subset\mathbb{R}^{d}$ is a bounded domain, the $G$-white noise on $\mathbf{H}$ does not have temporal Markov property. As introduced in Section~\ref{pre}, the canonical continuous Lévy process with independent and stationary temporal increment is the $G$-Brownian motion. Using this result and mimicking \cite{walsh}, we first define the sublinear expectation analog of the cylindrical Wiener process, then the $G$-spacetime white noise constructed in \cite{JP18} can be informally realized as its weak temporal derivative.

A function $X:\;\hat{\Omega}\to\mathbf{H}$ is called an $\mathbf{H}$-valued random variable if and only if for any $h\in\mathbf{H}$, $\langle X,h\rangle\in\hat{\mathcal{H}}$, and the family indexed by time $X(t)$ is then called an $\mathbf{H}$-valued stochastic process. In particular, $X$ is called $G$-Gaussian distributed on $\mathbf{H}$ if for all $h\in\mathbf{H}$, any 1-dimensional projection $\langle X,h\rangle$ is a mean-zero $G$-Gaussian random variable. When $\mathbf{H}=L^{2}(D)$, another way to understand this $1d$ marginal distribution is that $h(x)\in L^{2}(D)$ is integrated against the random function $X$. Such stochastic integration is essentially equivalent to the isometry of white noise, since given Theorem~\ref{dist}, we know that both of them can be calculated using orthonormal basis. Roughly speaking, for any fixed time, cylindrical $G$-Brownian motion essentially reduces to a $G$-white noise parametrized by $L^{2}(D)$, and the study of its temporal part mimics the derivation of Itô's calculus discussed in \cite{Peng19}. To begin the rigorous analysis, we first observe the following simple fact, the proof of which is an easy exercise.

\begin{lemma}\label{lem}
Let $\{X_{n}\}_{n\ge1}, X\in\hat{\mathcal{H}}_{1}$ such that $\lim_{n\to\infty}\hat{\mathbb{E}}\vert X_{n}-X\vert=0$, then there exists a subsequence $\{X_{n_{i}};i\in\mathbb{N}_{+}\}\subset\{X_{n};n\in\mathbb{N}_{+}\}$ that converges quasi-surely (q.s.) to $X$ as $i\to\infty$. 
\end{lemma}

Due to the divergence of variance, a cylindrical $G$-Brownian motion on $\mathbf{H}$ is in fact taking values in a larger Hilbert space $\mathbf{H}'$ via a densely defined embedding $\iota:\mathbf{H}\hookrightarrow\mathbf{H}'$. In the linear expectation case \cite{hairer}, $\mathbf{H}'$ typically lies in the family of interpolation spaces and its precise structure is sometimes irrelevant because infinite dimensional Gaussian measures are uniquely characterised by the Cameron-Martin spaces. There's a similar trouble for the Gaussian free field and in $2d$, this random generalized function is taking values in Sobolev spaces $H^{-s}(D)$, $s>0$. The ambient space $\mathbf{H}'$ does tell us about the spatial regularity of the random field, but before this, let's show the existence by a simple calculation. Since the embedding $\iota$ only depends on the choice of the sequence $a\in\ell^{2}(\mathbb{N})$, we denote the inner product (henceforth the norm) on $\mathbf{H}'$ by adding an $a$-subscript.

\begin{proposition}\label{prop:cylindrical}
    Fix any $a=(a_{n})_{n\ge1}\in\ell^{2}(\mathbb{N})$, and any orthonormal basis $\{e_{n}\}_{n\ge1}$ of $L^{2}(D)$, define the Hilbert-Schmidt embedding $\iota$: $L^{2}(D)\mapsto \mathbf{H}':=\{f;\sum_{n\ge1}a_{n}^{2}\langle f,e_{n}\rangle^{2}<\infty\}$ by
    \begin{equation}\label{embedd}
    \iota e_{n}(x)=e_{n}(x),\quad\iota f(x)=\sum_{n=1}^{\infty}a_{n}\langle f(x),e_{n}(x)\rangle (a_{n}^{-1}e_{n}(x)),\quad\forall f\in L^{2}(D).
    \end{equation}
    The identities in \eqref{embedd} should be understood in the space $\mathbf{H}'$, in which the orthonormal basis is given by $\{a_{n}^{-1}e_{n}\}_{n\geq1}$. As a consequence, there exist an $\mathbf{H}'$-valued cylindrical $G$-Brownian motion $\mathbb{W}(t)$, which is given by the expansion
    \begin{equation}\label{expand}
        \mathbb{W}(t)=\sum_{n=1}^{\infty}\mathbb{W}_{n}(t)\iota e_{n}(x),
    \end{equation}
    where $\{\mathbb{W}_{n}(t)\}_{n\ge1}$ is a sequence of i.i.d. $G$-Brownian motion. For each $t\ge0$, the series in \eqref{expand} converges q.s. on $\mathbf{H}'$, and for any $h,k\in\mathbf{H}'$ with $\langle h,k\rangle_{a}\ge0$ and $0\le s\le t$, the $G$-Gaussian random variables $\langle\W(t),h\rangle_{a}$ and $\langle\W(s),k\rangle_{a}$ satisfy the covariance bound:
    \begin{equation*}
    \E(\langle\W(t),h\rangle_{a}\langle\W(s),k\rangle_{a})\le s\wedge t\langle k,\iota\iota^{*}h\rangle_{a}\bar{\sigma}^{2},\quad-\E(-\langle\W(t),h\rangle_{a}\langle\W(s),k\rangle_{a})\ge s\wedge t\langle k,\iota\iota^{*}h\rangle_{a}\underline{\sigma}^{2}.
    \end{equation*}
    In the above, $\iota^{*}:\mathbf{H}'\mapsto\mathbf{H}$ denotes the adjoint operator.
\end{proposition}
\begin{proof}
    We only need to check the definition \eqref{expand} by direct calculation
    $$
    \lim_{n,m\to\infty}\hat{\mathbb{E}}\big(\sum_{i=n}^{m}a_{i}^{2}\mathbb{W}_{i}(t)^{2}\big)\le \lim_{n,m\to\infty}t\bar{\sigma}^{2}\sum_{i=n}^{m}a_{i}^{2}\to0.
    $$
    Hence, for any fixed time $t\geq0$, $\sum_{i\ge1}a_{i}^{2}\mathbb{W}_{i}(t)^{2}$ is Cauchy in Banach space $\tilde{\mathcal{H}}^{1}$. We claim that there exists an $\mathbb{X}(t)\in\tilde{\mathcal{H}}^{1}$ as the $L^{1}$ limit of the series.
    According to Lemma~\ref{lem}, we can find a subsequence converging q.s. to $\mathbb{X}(t)$. Since the series has non-negative entries, $\mathbb{X}(t)$ is in fact the q.s. limit of $\sum_{i\ge1}a_{i}^{2}\mathbb{W}_{i}(t)^{2}$, and $\mathbb{X}(t)=\Vert\mathbb{W}(t)\Vert_{a}^{2}$. Finally, the $G$-Gaussianity of $\W(t)$ is verified by the moment method established in Theorem~\ref{dist} and Remark~\ref{moment-method}.
    \end{proof}

Since the spatial randomness of $\W(t)$ is the white noise, we develope another construction in the following result, which shows the coherence of these two random fields. 

\begin{proposition}\label{another}
Given a $G$-spatial white noise $\{\mathbb{W}(f(x));f(x)\in L^{2}(D)\}$ on $(\hat{\Omega},\hat{\mathcal{H}},\hat{\mathbb{E}})$, there exist a family of sublinear expectation spaces $(\hat{\Omega}_{t},\hat{\mathcal{H}}_{t},\hat{\mathbb{E}}_{t})_{t\ge0}$ such that for any fixed $t\ge0$, we have a G-white noise $\W(\sqrt{t}\;\cdot\;)$ on $(\hat{\Omega}_{t},\hat{\mathcal{H}}_{t},\hat{\mathbb{E}}_{t})$ with covariance functional: 
$$
G_{f_{1},...,f_{n}}(A)=\frac{1}{2}\displaystyle\sup_{\theta\in[t\underline{\sigma}^{2}, t\bar{\sigma}^{2}]}\tr(AB),\quad\forall A\in\mathbb{S}(n),\;B\in\Theta,\;f_{1},...,f_{n}\in L^{2}(D).
$$
In particular, $\Theta=\big(\int_{D} \theta f_{i}(x)f_{j}(x)dx\big)_{i,j=1}^{n}$.
Moreover, there exists a product space (to be defined in the proof) and a cylindrical G-Brownian motion $W(t)$ such that $\W(\sqrt{t}f(x))\overset{d}{=}\langle \W(t),f\rangle_{a}$.
\end{proposition}
 
\begin{proof}
To begin the construction, fix $t=1$ and find the smallest sub-vector lattice of $\hat{\mathcal{H}}$ spanned by the given white noise $\{\mathbb{W}(f(x));\forall f(x)\in 
 L^{2}(D)\}$ and denote it by $\hat{\h}_{1}$. Next we would like to perform reconstructions to the sublinear expectation spaces via a family of transformations $\tau_{t}:\:(\hat{\Omega},\hat{\mathcal{H}}_{1},\hat{\mathbb{E}})\mapsto(\hat{\Omega}_{t},\hat{\mathcal{H}}_{t},\hat{\mathbb{E}}_{t})$. For any $X_{1},...,X_{n}\in\hat{\mathcal{H}}_{1}$ and $\varphi\in C_{l,Lip}(\mathbb{R}^{n})$, define
\begin{equation*}
\hat{\Omega}_{t}=\hat{\Omega},\:\:\hat{\h}_{t}=\hat{\mathcal{H}}_{1},\:\:\E_{t}(\varphi(X_{1},...,X_{n}))=\hat{\mathbb{E}}(\sqrt{t}\varphi(X_{1},...,X_{n})).
\end{equation*}
Moreover, for different values of $t\ge0$ we require these sublinear expectation spaces to be mutually independent of each other. If we naively copy (just algebraically) the functions in the family $\W(\cdot)$ and paste them on the space $(\hat{\Omega}_{t},\hat{\mathcal{H}}_{t},\hat{\mathbb{E}}_{t})$ then its easy to see that the couple $(\W(\cdot),\E_{t})$ is identically distributed to $(\W(\sqrt{t}\;\cdot\;),\E)$. The final step is to appropriately glue $\{(\hat{\Omega}_{t},\hat{\mathcal{H}}_{t},\E_{t}),t\ge0\}$ together to obtain the distribution of a cylindrical $G$-Brownian motion. 

For this product space, we first let $\Omega=\otimes_{t\ge0}\hat{\Omega}_{t}$. The space of all sample paths $\h=Lip(\hat{\h})$ is defined very naturally by cylinder functions
$$
    Lip(\hat{\h})=\cup_{n=1}^{\infty}Lip(\hat{\h}_{n})
$$$$   
    =\bigcup_{n=1}^{\infty}\{\varphi(X_{1}\circ\pi_{t_{1}},...,X_{m}\circ\pi_{t_{m}});\forall\varphi\in C_{l,Lip}(\mathbb{R}^{m}), X_{i}\in\hat{\mathcal{H}}_{t_{i}}=\hat{\mathcal{H}}_{1}, 0\le t_{1}...\le t_{m}\le n,m\in\mathbb{N} \}
$$
where $\pi_{t}:\Omega\mapsto\hat{\Omega}_{t}$ is the coordinate projection. For the definition of sublinear expectation $\mathbb{E}$, extract any random variable $\varphi(X_{1}\circ\pi_{t_{1}},...,X_{m}\circ\pi_{t_{m}})$ and compute by independence
\begin{align}
\label{markov}
    \mathbb{E}(&\varphi(X_{1}\circ\pi_{t_{1}},...,X_{m}\circ\pi_{t_{m}}))\notag\\
    &=\mathbb{E}[\hat{\mathbb{E}}_{t_{m}-t_{m-1}}(\varphi(x_{1},...,x_{m-1},X))_{x_{1}=X_{1}\circ\pi_{t_{1}},...,x_{m-1}=X_{m-1}\circ\pi_{t_{m-1}}}]\notag\\
    &=\E_{t_{1}}\circ\E_{t_{2}-t_{1}}\circ\cdot\cdot\cdot\circ\E_{t_{m}-t_{m-1}}(\varphi(X_{1}\circ\pi_{t_{1}},...,X_{m}\circ\pi_{t_{m}})).\tag{Markov}
\end{align}
We claim that the random field $W(t):=\{X\circ\pi_{t};X\in\W(\cdot),t\ge0\}$ is the desired cylindrical $G$-Brownian motion. Indeed, one easily checks that for each fixed $t\ge0$, it is a $G$-spatial white noise with the correct covariance functional, and the orthonormal expansion \eqref{expand} is obtained by applying the Hilbert-Schmidt embedding $\iota$. Last, \eqref{markov} guarantees that for any $f\in L^{2}(D)$, the $1d$ process $\langle W(t),f\rangle_{a}$ is a $G$-Brownian motion.
\end{proof}

Consider a spacetime random noise $f(t,x,\omega)$ with $t\ge0$, $x\in D$ and $\omega\in\hat{\Omega}$, the spatial and temporal part of $f$ is understood to be a random generalized function in the sense that, for any deterministic function $g(t,x)\in L^{2}([0,\infty)\times D)$, we have $\langle f,g\rangle\in \tilde{\h}^{2}$. If for any fixed $t\in[0,T]$, the field $f(t,\cdot,\omega)\subset Lip(\hat{\h}_{t})$, and satisfies $(\hat{\mathbb{E}} \int_{0}^{t}\Vert f(s,x,\omega)\Vert^{2}_{L^{2}}ds)^{1/2}<\infty$, we say $f(t,x,\omega)$ is a predictable $\h^{2}([0,T];L^{2}(D))$ process. For such processes, we now define the infinite dimensional stochastic integration against the cylindrical $G$-Brownian motion $\W(t)$.
\begin{proposition}
    Let $f(t,x,\omega)\in \h^{2}([0,T];L^{2}(D))$, we define the stochastic integral of $f$ with respect to the cylindrical $G$-Brownian motion \eqref{expand} in the following sense:
    $$
    \int_{0}^{t}f\cdot\mathbb{W}(ds):=\int_{0}^{t}\langle f,\mathbb{W}(ds)\rangle=\sum_{i=1}^{\infty}\int_{0}^{t}\langle f(s,x),\iota^{*}\iota e_{i}(x)\rangle\mathbb{W}_{i}(ds),\quad t\in[0,T].
    $$
    In fact, the integration is a bounded linear map from $\h^{2}([0,T];L^{2}(D))$ to $\tilde{\h}^{2}$, which is also strictly bounded from below:
    \begin{equation}\label{up}
        \hat{\mathbb{E}}\Big\vert\int_{0}^{T}f\cdot\mathbb{W}(ds)\Big\vert^{2}\lesssim\bar{\sigma}^{2}\hat{\mathbb{E}} \int_{0}^{T}\Vert f(t,x,\omega)\Vert^{2}_{L^{2}}dt,
    \end{equation}
    \begin{equation}\label{down}
        \hat{\mathbb{E}}\Big\vert\int_{0}^{T}f\cdot\mathbb{W}(ds)\Big\vert^{2}\gtrsim\underline{\sigma}^{2}\hat{\mathbb{E}} \int_{0}^{T}\Vert f(t,x,\omega)\Vert^{2}_{L^{2}}dt.
    \end{equation}
\end{proposition}
\begin{proof}
    Since $\mathbb{W}(t)$ is white in space, we proceed to calculate
    \begin{align*}
    &\hat{\mathbb{E}}\Big[\sum_{i=n}^{m}\int_{0}^{t}\langle f(s,x),\iota^{*}\iota e_{i}(x)\rangle\mathbb{W}_{i}(ds)\Big]^{2}
    =\hat{\mathbb{E}}\sum_{i=n}^{m}\Big[\int_{0}^{t}\langle f(s,x),\iota^{*}\iota e_{i}(x)\rangle\mathbb{W}_{i}(ds)\Big]^{2}\\
    &\lesssim\hat{\mathbb{E}}\Big(\xi+\bar{\sigma}^{2}\sum_{i=n}^{m}\int_{0}^{t}\vert\langle f(s,x),e_{i}(x)\rangle\vert^{2}ds\Big)
    \le\bar{\sigma}^{2}\hat{\mathbb{E}}\sum_{i=n}^{m}\Big(\int_{0}^{t}\vert\langle f(s,x),e_{i}(x)\rangle\vert^{2}ds\Big).
    \end{align*}
    where the positive constant in the `$\lesssim$' notion is arbitrary and only depends on the choice of the embedding $\iota$. In the above calculation, we have introduced an error random variable $\xi$ and according to the quadratic variation of $G$-Brownian motion (see \cite[Chapter~3]{Peng19}), it satisfies the following estimate
    \begin{equation*}
    \xi=\sum_{i=n}^{m}\Big[\big(\int_{0}^{t}\langle f(s,x),e_{i}(x)\rangle\mathbb{W}_{i}(ds)\big)^{2}-\bar{\sigma}^{2}\int_{0}^{t}\vert\langle f(s,x),e_{i}(x)\rangle\vert^{2}ds\Big],
    \end{equation*}
    \begin{equation*}
    \hat{\mathbb{E}}\xi\le\sum_{i=n}^{m}\hat{\mathbb{E}}\Big[\big(\int_{0}^{t}\langle f(s,x),e_{i}(x)\rangle\mathbb{W}_{i}(ds)\big)^{2}-\bar{\sigma}^{2}\int_{0}^{t}\vert\langle f(s,x),e_{i}(x)\rangle\vert^{2}ds\Big]\le0.
    \end{equation*}
    This shows the integral $\int_{0}^{t}f\cdot\mathbb{W}(dt)$ is well-defined in the $L^{2}$ sense, meanwhile the continuity and boundedness of the integral operator is established. For stochastic calculus with uncertainty, the Itô isometry also has a lower bound \eqref{down} (also recall \eqref{covariance-bd}), and for the proof, we construct another handy random variable $\xi'$ in the following sense:
    \begin{equation*}
    \xi'=\sum_{i=n}^{m}\Big[\underline{\sigma}^{2}\int_{0}^{t}\vert\langle f(s,x),e_{i}(x)\rangle\vert^{2}ds-\big(\int_{0}^{t}\langle f(s,x),e_{i}(x)\rangle\mathbb{W}_{i}(ds)\big)^{2}\Big],
    \end{equation*}
    \begin{equation*}
    \hat{\mathbb{E}}\xi'\le\sum_{i=n}^{m}\hat{\mathbb{E}}\Big[\underline{\sigma}^{2}\int_{0}^{t}\vert\langle f(s,x),e_{i}(x)\rangle\vert^{2}ds-\hat{\mathbb{E}}\big(\int_{0}^{t}\langle f(s,x),e_{i}(x)\rangle\mathbb{W}_{i}(ds)\big)^{2}\Big]\le0.
    \end{equation*}
    These yields
    \begin{align}
        \hat{\mathbb{E}}\Big\vert\int_{0}^{T}f\cdot\mathbb{W}(ds)\Big\vert^{2}
        \lesssim\lim_{n\to\infty}\hat{\mathbb{E}}\Big(\sum_{i=1}^{n}\bar{\sigma}^{2}\int_{0}^{T}\vert\langle f(s,x),e_{i}(x)\rangle\vert^{2}ds\Big),\label{551}\\
        \hat{\mathbb{E}}\Big\vert\int_{0}^{T}f\cdot\mathbb{W}(ds)\Big\vert^{2}
        \gtrsim\lim_{n\to\infty}\hat{\mathbb{E}}\Big(\sum_{i=1}^{n}\underline{\sigma}^{2}\int_{0}^{T}\vert\langle f(s,x),e_{i}(x)\rangle\vert^{2}ds\Big).\label{552}
    \end{align}
    In addition, we have
    $$
    0\le\hat{\mathbb{E}}\int_{0}^{T}\Vert f(t,x,\omega)\Vert^{2}_{L^{2}}dt-\hat{\mathbb{E}}\Big(\sum_{i=1}^{n}\int_{0}^{T}\vert\langle f(s,x),e_{i}(x)\rangle\vert^{2}ds\Big)
    $$
    $$
    \le\hat{\mathbb{E}}\Big(\int_{0}^{T}\Vert f(t,x,\omega)\Vert^{2}_{L^{2}}dt-\sum_{i=1}^{n}\int_{0}^{T}\vert\langle f(s,x),e_{i}(x)\rangle\vert^{2}ds\Big)\to0,\quad n\to\infty.
    $$
Thus the limit and the sublinear expectation in \eqref{551} and \eqref{552} are exchangeable and the proof is therefore complete. 
\end{proof}

\section{Stochastic quantization and massive $G$-Gaussian free field}
\label{sq}

Consider the path integral quantization of a bosonic field $\phi$ in 4-dimensional spacetime, which is based on a measure informally given by $\exp[iS(\phi)]\mathcal{D}\phi$, where $S(\cdot)$ is the action functional and $\mathcal{D}(\phi)$ is the uniform distribution on the set of all classical fields $\phi\in\mathcal{S}'(\mathbb{R}^{4})$. The rigorous construction of this measure faces a major difficulty, that is, the Lorentzian nature of the spacetime metric leads to an oscillating kinetic term $\exp[iS(\phi)]$. One possible solution to this problem is the Wick rotation. Namely, we rotate the temporal axis counterclockwise by $\pi/2$, and the result is an infinite volume Gibbs measure:
\begin{equation}\label{EQFT}
    \mu(d\phi)\propto e^{-S_{E}(\phi)}\mathcal{D}(\phi).
\end{equation}
The positive definiteness of the Riemannian metric and the $E$-subscript suggested that this model is an Euclidean quantum field theory (EQFT). Canonical constructions of measures of the type \eqref{EQFT} is required to satisfy the Osterwalder-Schrader axioms (see \cite{GJ}), and luckily, much of the scalar field theory can be resolved. For example, see \cite{salmhofer} for a detailed discussion of the renoramlization group (RG) approach to the $\phi^{4}$ theory.

For the Euclidean bosonic field on $D\subset\mathbb{R}^{d}$, the measure is expressed by
\begin{equation}\label{bos}
    \mu(d\phi)=\frac{1}{Z}e^{-\frac12\langle\phi,(-\Delta+m^{2})\phi\rangle-\lambda\int_{D}V(\phi(x))dx}\mathcal{D}(\phi).
\end{equation}
Here $Z$ is the partition function, $\langle\cdot,\cdot\rangle$ denotes the usual $L^{2}$ inner product, $m$ denotes the mass of the field quanta, and $V$ is the interacting potential. In 2-dimensions, \eqref{bos} is called the $\phi^{4}_{2}$ model and the Sine-Gordon model if we take $V(\phi)=\phi^{4}$ and $V(\phi)=\cos(\beta\phi)$, $\beta^{2}<8\pi$ respectively. To simplify the argument, we only consider the free field case, i.e. $V=0$, and the model is nothing but a infinite dimensional Gaussian measure with covariance operator $-\Delta+m^{2}$. 

Now we briefly explain Parisi-Wu's dynamical approach for the construction of \eqref{EQFT}. In addition to the spacetime coordinates $x\in D$, we construct an additional fictitious time $t\ge0$ and couple it to the field: $\phi(x)\mapsto\phi(x,t)$. The physical interpretation of this trick is that the quantum system is under an imaginary coupling with a heat reservoir at a fixed large temperature, and $t$ measures the time to set up thermodynamical equilibrium.
Similar to a classical particle immersed in fluid, this fictitious evolution $\phi(x,t)$ is described by a Langevin dynamics, which is an SPDE of the following type:
\begin{equation}\label{1111}
\partial_{t}\phi(x,t)=-\frac{\delta S_{E}}{\delta\phi(x,t)}+W (x,t).
\end{equation}
The notion $W(x,t)$ refers to the spacetime Gaussian white noise on $[0,\infty)\times D$, and in linear expectation theory, it is a centered Gaussian random field with covariance $\mathbb{E}[W(x,t)W(y,s)]=\delta(x-y)\delta(t-s)$. For the reason of this choice, one can think of the spatially discretized models, e.g., a lattice spin system on the rescaled space $\varepsilon\mathbb{Z}^{d}\cap D$. By graphical construction, the dynamics on the vertices is the random shift of the particle numbers driven by a family of spatially independent Poisson clocks, which approximates white noise in the scaling limit $\varepsilon\to0$. On the RHS of \eqref{1111}, we have a variational derivative of the Euclidean action $S_{E}=\int dx\mathcal{L}(\phi(x),\nabla\phi(x))$ with Lagrangian density $\mathcal{L}$. In the case of Gaussian free field (GFF), the Langevin dynamics on $D$ is the stochastic reaction-diffusion equation:
\begin{equation}\label{lan}
    \partial_{t}\phi(x,t)=(\Delta-m^{2})\phi(x,t)+W(x,t).
\end{equation}
It is well-known in stochastic analysis that the equilibrium measure $\mu$ of \eqref{lan} is the rigorous setting of massive Gaussian free field. The solution to the SPDE is the infinite dimensional Ornstein-Uhlenbeck process with Markov semigroup $P_{t}$. Starting from any initial distribution $\nu$, we have an exponentially fast convergence of measure $\Vert P_{t}^{*}\nu-\mu\Vert_{TV}$ in the total variation distance.

In the context of sublinear expectation, one expects a similar type of construction of the $G$-Gaussian free field ($G$GFF). However, the convergence to equilibrium is only known in a much weaker sense. Since the solution is a continuous process such that for each fixed time $t>0$, the distribution is zero mean $G$-Gaussian, we follow the moment method and expect the limit random field to be also $G$-Gaussian distributed. In particular, we want to specified the family of equilibrium covariance functionals. Following \cite{LS}, we will show that for any suitable test function $f$ and $g$, the large time limit obeys
\begin{equation}\label{ub}
    \lim_{t\to\infty}\E\big(\langle \phi(x,t),f(x)\rangle\langle \phi(y,t),g(y)\rangle\big)\lesssim\bar{\sigma}^{2}\iint_{D^{2}}G_{m}(x,y)f(x)g(y)dxdy,
\end{equation}
\begin{equation}\label{lb}
    \lim_{t\to\infty}-\E\big(-\langle \phi(x,t),f(x)\rangle\langle \phi(y,t),g(y)\rangle\big)\gtrsim\underline{\sigma}^{2}\iint_{D^{2}}G_{m}(x,y)f(x)g(y)dxdy.
\end{equation}
In the above, $G_{m}(x,y)$ denotes the massive Green's function of the operator $(2\pi)^{-1}(-\Delta+m^{2})$ in domain $D$. Due to the spectral property of the Laplacian on $D$, the bounded case and unbounded case are treated differently. Following \cite{sheffield}, we also expect a conformal symmetry of the geometry of the $2$-dimensional $G$GFF. In the following explanations, we assume a priori that $D$ is bounded.

\subsection{Dynamical $G$-Gaussian free field}

In this subsection, we study the solution to the stochastic PDE on bounded domain $D\subset\mathbb{R}^{d}$, $d\ge2$ with zero boundary condition:
\begin{equation}\label{spde}
    d\phi(x,t)=(\Delta-m^{2})\phi(x,t)dt+d\W(t),\quad \phi(x,0)=\psi(x)\in C_{0}^{\infty}(D),\; \phi(x,t)\vert_{x\in\partial D}=0,
\end{equation}
where $\W(t)$ is defined by \eqref{expand}. A predictable square-integrable process $\phi(x,t)\in\h^{2}([0,T];L^{2}(D))$ is called a mild solution to \eqref{spde} if it satisfies the stochastic Duhamel's principle
\begin{equation}\label{duhamel}
    \phi(x,t)=\int_{D}P_{t}(x,y)\psi(y)dy+\int_{0}^{t} P_{t-s}(x,y)\cdot \W(ds,y).
\end{equation}
For the analytic semigroup $P_{t}$ generated by $\mathcal{L}=\Delta-m^{2}$, we refer to the integral kernel
\begin{equation}\label{heat}
    P_{t}(x,y)=\frac{1}{(4\pi t)^{d/2}}\;e^{-\frac{\vert x-y\vert^{2}}{4t}-m^{2}t}.
\end{equation}
Recall Propostition~\ref{another}, the dot notation on the RHS of \eqref{duhamel} is a shorthand for the spatial stochastic integration with respect to the fixed-time $G$-white noise $\W(\sqrt{s}\;\cdot\;)$. We should mention that, the notion of weak solution discussed in \cite{walsh} actually coincide with mild solutions in most cases, the proof of this equivalence is similar to the linear expectation case \cite{hairer}.

The main goal of this article is to show that the mild solution \eqref{duhamel} converges in law to the massive $G$GFF under $t\to\infty$. We now give a definition of the massive $G$GFF on $D$ with Dirichlet boundary condition, which is motivated by the linear expectation case discussed in \cite{BP}.

\begin{definition}\label{GGFF}
    Consider a $G$-Gaussian random field $\Psi$ parametrized by the family $\mathcal{M}$ of compactly supported finite signed measures on $D$. Then $\Psi$ is called a Dirichlet $G$-Gaussian free field with mass $m>0$ and variance regime $0<\underline{\sigma}^{2}\leq\bar{\sigma}^{2}$ if for any $\mu,\nu\in\mathcal{M}$, we have
    \begin{equation}\label{cor-GGFF}
        \E(\Psi_{\mu}\Psi_{\nu})=\bar{\sigma}^{2}\int_{D^{2}}G_{m}(x,y)\mu(dx)\nu(dy),\quad-\E(-\Psi_{\mu}\Psi_{\nu})=\underline{\sigma}^{2}\int_{D^{2}}G_{m}(x,y)\mu(dx)\nu(dy).
    \end{equation}
    The function $G_{m}(x,y)$ in \eqref{cor-GGFF} refers to the Green's function for $(2\pi)^{-1}(\Delta-m^{2})$.
\end{definition}

\begin{remark}\label{rem:GGFF}
    The above definition does not yield any uniqueness of such massive $G$GFF. By taking the large time limit for \eqref{duhamel}, we show that the limit distribution satisfies Definition~\ref{GGFF}, which only guarantees the existence of the $G$GFF. In our setting, we restrict our argument to the family of measures: $\{\mu\in\mathcal{M};\mu(dx)=\rho(x)dx,\;\rho\in H^{-1}_{0}(D)\}$.
\end{remark}

Before we could specify the state space of the mild solution of \eqref{spde}, we need some preliminaries on the spectrum of finite volume Laplacian. Observe that the $L^{2}$ space of closed interval $[0,2\pi]$ admits an orthonormal basis:
$$
\big\{\frac{1}{\sqrt{2\pi}},\frac{\sin{kx}}{\sqrt{\pi}},\frac{\cos{kx}}{\sqrt{\pi}},k\in\mathbb{N}\big\}.
$$
One can check that each one of these trigonometric function is an eigenfunction of the 1-dimensional operator $\partial_{x}^{2}$. In fact, one can extend this to any bounded curved spacetime. For example, \cite{berard} studied the the following Dirichlet boundary value problem (Problem D) for the Laplace-Beltrami operator $-\Delta f=-\operatorname{div}(\operatorname{grad}f)=-\nabla^{i}\nabla_{i}f$ defined on a compact Riemannian manifold $(M,g)$:
\begin{equation}\label{D}\tag{D}
    -\Delta f=\lambda f,\quad f\in C^{\infty}(M),\;f\vert_{\partial M}=0.
\end{equation}
\begin{lemma}\label{lem:lap}
    Problem~\eqref{D} admits the following solution:
    \begin{enumerate}[(i)]
    \item The set of eigenvalues consists of an infinite sequence $\mathrm{0<\lambda_{1}\le\lambda_{2}\le\cdot\cdot\cdot}$ such that $\lim_{n\to\infty}\lambda_{n}=\infty$.

    \item Each eigenvalue has finite multiplicity and the eigenspaces for distinct eigenvalues are orthogonal in $\mathbf{H}=L^{2}(M)$.

    \item Each eigenfunction is smooth, analytic and all eigenfunctions span a dense linear subspace of $\mathbf{H}=L^{2}(M)$.
    \end{enumerate}
\end{lemma}

For simplicity, we relabel the set of eigenvalues by increasing order $\{\lambda_{n}, n\in\mathbb{N}\}$, and denote the set of orthonormal eigenfunctions by $\{e_{n}(x)\in C^{\infty}(D), n\in\mathbb{N}\}$. By Weyl's law, the asymptotic distribution is roughly given by $\lambda_{n}\sim n^{2/d}$. In order to expand the cylindrical $G$-Brownian motion $\W(t)$, henceforth the mild solution $\phi(x,t)$ by eigenfunctions, we choose the canonical embedding to be $\iota:L^{2}(D)\hookrightarrow H^{s}_{0}(D)$ with $s=-d/2-\varepsilon$ and $\varepsilon>0$ arbitrarily small. The fact that $\iota$ is Hilbert-Schmidt follows from $\{\lambda_{n}^{-s/2}e_{n}\}$ being the orthonormal basis of $H^{s}_{0}(D)$. As a consequence, for any fixed $t\ge0$, the mild solution also has spatial regularity: $\phi(x,t)\in H^{s}_{0}(D)$, i.e., we have $\sum_{n\ge1}n^{2s/d}\langle\phi(t),e_{n}\rangle^{2}<\infty$, q.s..
Heuristically speaking, we imagine that $\phi(x,t)=\sum_{n\geq1}\phi_{n}(t)e_{n}(x)$ converges in $H^{s}_{0}(D)$, and taking the $L^{2}(D)$ inner product will give:
\begin{equation*}
\langle(\Delta-m^{2})\phi(x,t),e_{n}(x)\rangle=-(\lambda_{n}+m^{2})\phi_{n}(t).
\end{equation*}
Notice that $\phi_{n}(t)$ is a sequence of predictable diffusion processes satisfying SDEs driven by i.i.d. $1$-dimensional $G$-Brownian Motions:
\begin{equation}\label{ou}
    d\phi_{n}(t,\omega)=-(\lambda_{n}+m^{2})\phi_{n}(t,\omega)dt+\mathbb{W}_{n}(dt),\;n\in\mathbb{N}
\end{equation}
with $\mathbb{W}_{n}(1)\backsim\mathcal{N}(0, [\underline{\sigma}^{2},\bar{\sigma}^{2}])$. Calculating the solution to these equations implies that the Fourier coefficients $\phi_{n}$ are indeed the $G$-Ornstein Uhlenbeck processes. Summing up $n\geq1$ in the Sobolev space $H^{s}_{0}(D)$ will produce the full solution to \eqref{spde}. In the following, we begin the proof of our heuristics with the property of the $G$-Ornstein Uhlenbeck process.

\begin{lemma}\label{ouprocess}
    Without loss of generality, let $\phi(t)$ denote a $G$-Ornstein Uhlenbeck process expanded by the Itô's integral with respect to the $G$-Brownian motion $\W(dt)$:
    $$
    \phi(t)=ce^{-at}+\int_{0}^{t}e^{-a(t-s)}\mathbb{W}(ds),\quad t\in[0,T],\;a>0.
    $$
    Then the covariance functions $\Gamma(s,t),\forall s,t\in[0,T]$ obeys the following estimation:
    \begin{align}
\Gamma(s,t)=\mathrm{cov}(\phi(t),\phi(s))\le\frac{\bar{\sigma}^{2}}{2a}(e^{-a\vert t-s\vert}-e^{-a(s+t)}).
    \end{align}
\end{lemma}
\begin{proof}
    By calculation we have 
    \begin{align*}
        \Gamma_{n}(s,t)
        &=e^{-at}e^{-as}\hat{\mathbb{E}}\int_{0}^{t}e^{au}\mathbb{W}(du)\int_{0}^{s}e^{ar}\mathbb{W}(dr)\\
        &=e^{-at}e^{-as}\hat{\mathbb{E}}\int_{0}^{T}e^{2ar}1_{0\le r\le t}1_{0\le r\le s}\langle\mathbb{W}\rangle(dr)\\
        &\le\bar{\sigma}^{2}e^{-at}e^{-as}\hat{\mathbb{E}}\int_{0}^{s\wedge t}e^{2ar}dr\\
        &=\frac{\bar{\sigma}^{2}}{2a}(e^{-a\vert t-s\vert}-e^{-a(s+t)}).
    \end{align*}
The second equality follows from Itô's isometry and the third inequality follows from the property of integration against quadratic variation process (see Lemma 3.4.3 in \cite{Peng19}).
\end{proof}

\begin{theorem}
For any $T>0$, the stochastic reaction-diffusion equation \eqref{spde} admits a mild solution $\phi\in\h^{2}([0,T];H_{0}^{s}(D))$. More precisely, for any initial distribution $\psi(x)\in C_{0}^{\infty}(D)$, we have the decomposition:
\begin{equation}\label{mildd}
\phi(x,t)=(P_{t}\psi)(x)+\sum_{n=1}^{\infty}\phi_{n}(t)e_{n}(x),\quad q.s.\;\mathrm{in\;} H_{0}^{s}(D).
\end{equation}
In particular, $s<-d/2$, the semigroup $P_{t}$ is given by \eqref{heat}, and for each $n\ge1$, the $G$-Ornstein Uhlenbeck process $\phi_{n}(t)$ is the solution to \eqref{ou} with $a=\lambda_{n}+m^{2}$.
    
\end{theorem}

\begin{proof}
    We now complete the details of the proof sketched in the previous contexts. To calculate the mild solution \eqref{duhamel}, for each $n\ge1$, we take the $H^{s}_{0}(D)$ inner product $\lambda_{n}^{s}\langle\;\cdot\;,\lambda_{n}^{-s/2}e_{n}(x)\rangle$ on both sides and then eliminate the extra $\lambda_{n}^{s/2}$. Since the orthonormal expansion and the summation over $n$ does not affect the deterministic part of the dynamics $P_{t}\psi$, the mild solution equals to \eqref{mildd} in $\h^{2}([0,T];H_{0}^{s}(D))$ if and only if one has existence and uniqueness for the solution of the SDEs \eqref{ou} in the space $\h^{2}([0,T];\mathbb{R})$. The convergence of the series $\sum_{n\geq1}\phi_{n}(t)e_{n}(x)$ in $H^{s}_{0}(D)$ is ensured by Lemma~\ref{ouprocess} and the fact that $\sum_{n\geq1}\lambda_{n}^{s}\bar{\sigma}^{2}(\lambda_{n}+m^{2})^{-1}<\infty$.
    
    The topology of the space of square integrable predictable processes $\eta\in\h^{2}([0,T];\mathbb{R})$ is generated by the norm $[\int_{0}^{2\pi}\hat{\mathbb{E}}(\eta_{t}^{2})dt]^{2}$. We define a family of mapping $\Lambda^{n}:\h^{2}([0,T];\mathbb{R})\mapsto\h^{2}([0,T];\mathbb{R})$ by 
    \begin{equation}
       \Lambda^{n}_{t}(X)=\psi_{n}-\int_{0}^{t}(\lambda_{n}+m^{2})Xds+\int_{0}^{t}\mathbb{W}_{n}(ds),\quad \psi_{n}=\lambda_{n}^{-s/2}\langle\psi,e_{n}\rangle,\;\forall X\in\h^{2}([0,T];\mathbb{R}).
    \end{equation}
    For any $X_{1}$, $X_{2}\in\h^{2}([0,T];\mathbb{R})$, we compute
    \begin{align*}
        \hat{\mathbb{E}}\vert&\Lambda^{n}_{t}(X_{1})-\Lambda^{n}_{t}(X_{2})\vert^{2}
        =\hat{\mathbb{E}}\Big\vert\int_{0}^{t}(\lambda_{n}+m^{2})(X_{1}-X_{2})ds\Big\vert^{2}\\
        &\le(\lambda_{n}+m^{2})^{2}\hat{\mathbb{E}}\int_{0}^{t}(X_{1}-X_{2})^{2}ds\le(\lambda_{n}+m^{2})^{2}\int_{0}^{t}\hat{\mathbb{E}}(X_{1}-X_{2})^{2}ds.
    \end{align*}
    Multiplying both sides by $e^{-2(\lambda_{n}+m^{2})^{2}t}$ and integrate over $t$, we get
   \begin{align*}
   \int_{0}^{2\pi}e^{-2(\lambda_{n}+m^{2})^{2}t}&\hat{\mathbb{E}}\vert\Lambda^{n}_{t}(X_{1})-\Lambda^{n}_{t}(X_{2})\vert^{2}dt\\
   &\le(\lambda_{n}+m^{2})^{2}\int_{0}^{2\pi}\int_{0}^{t}e^{-2(\lambda_{n}+m^{2})^{2}t}\hat{\mathbb{E}}(X_{1}-X_{2})^{2}dsdt\\
       & =(\lambda_{n}+m^{2})^{2}\int_{0}^{2\pi}\int_{s}^{2\pi}e^{-2(\lambda_{n}+m^{2})^{2}t}\hat{\mathbb{E}}(X_{1}-X_{2})^{2}dtds\\
     &=\frac{1}{2}\int_{0}^{2\pi}\big[e^{-2(\lambda_{n}+m^{2})^{2}s}-e^{-4(\lambda_{n}+m^{2})^{2}\pi}\big]\hat{\mathbb{E}}(X_{1}-X_{2})^{2}ds\\
        &\le\frac{1}{2}\int_{0}^{2\pi}e^{-2(\lambda_{n}+m^{2})^{2}s}\hat{\mathbb{E}}(X_{1}-X_{2})^{2}ds.
\end{align*}
 Notice that $[\int_{0}^{2\pi}\hat{\mathbb{E}}(\eta_{t}^{2})dt]^{2}$ and $[\int_{0}^{2\pi}e^{-2(\lambda_{n}+m^{2})^{2}t}\hat{\mathbb{E}}(\eta_{t}^{2})dt]^{2}$ are two equivalent norms, $\Lambda^{n}$ is a contraction on Banach space $\h^{2}([0,T];\mathbb{R})$, this concludes the proof.
\end{proof}

\begin{proof}[Proof of Theorem~\ref{main}]
    Due to the Hilbert-Schmidt embedding and the moment method established in Theorem~\ref{dist} and Remark~\ref{moment-method}, the convergence in law and the $G$-Gaussianity of the limit random field are similarly deduced. Let $f_{n}=\lambda_{n}^{-1/2}e_{n}$ be the orthonormal basis of the Sobolev space $H^{1}_{0}(D)$ with a mass-corrected inner product $\langle f_{1},f_{2}\rangle_{\nabla,m}:=\langle (\Delta-m^{2})^{1/2} f_{1},(\Delta-m^{2})^{1/2} f_{2}\rangle$. Note that this inner product $\langle \cdot,\cdot\rangle_{\nabla,m}$ generates the same $H^{1}_{0}(D)$ topology as the usual one $\langle \nabla\cdot,\nabla\cdot\rangle$. The expansion \eqref{mildd} had shown that if we replace $e_{n}$ by $f_{n}$, the mild solution $\phi$ is quasi-surely well-defined in $H^{1+s}_{0}(D)$. However, from the definition of the Euclidean bosonic field \eqref{bos}, we expect $\phi(t,\cdot)$ to have regularity similar to that of the $H^{1}_{0}(D)$ functions. More precisely, in a weaker sense than q.s. convergence, $\phi(t,\cdot)$ is a $G$-Gaussian random field parametrized by $H^{-1}_{0}(D)$ via the testing:
    \begin{equation}\label{7}
        [\phi(t),f]=[ P_{t}\psi,f]+\sum_{n\ge1}\phi_{n}(t)[ f_{n},f],\quad\forall f\in H^{-1}_{0}(D).
    \end{equation}
    In particular, the bracket quadratic form denotes the $L^{2}$ pairing between an $H^{1}_{0}(D)$ function $f$ and an $H^{-1}_{0}(D)$ function $g$: 
    \begin{equation}\label{the[]}
    [f,g]:=\langle (\Delta-m^{2})^{1/2}f,(\Delta-m^{2})^{-1/2}g\rangle.
    \end{equation}
    For simplicity of notations, we abbreviate the LHS of \eqref{7} by $\phi_{f}(t)$, thus
    \begin{align}\label{lim}
        \lim_{t\to\infty}\E\big(\phi_{f}(t)\phi_{g}(t)\big)&\le\lim_{t\to\infty}\sum_{j,k\ge1}\E\big(\phi_{j}(t)\phi_{k}(t)\big[ f_{j},f] [f_{k},g]\big)\notag\\
        &\le\frac{\bar{\sigma}^{2}}{2(\lambda_{1}+m^{2})}\iint_{D^{2}}G_{m}(x,y)f(x)g(y)dxdy.
    \end{align}
    The first inequality follows from the exponential decay of heat kernel \eqref{heat} and the subadditivity of expectation. The second bound \eqref{lim} is a direct consequence of Lemma~\ref{ouprocess} and the positive spectral gap of the finite volume Laplacian $-\Delta$, see also Lemma~\ref{lem:lap}. The covariance lower bound is similarly derived. In conclusion, the proof is completed by choosing appropriate $\alpha=\alpha(d,D,m)>0$ which only depends on the geometric setting and the mass of the quantum bosonic free field.
\end{proof}

\bibliographystyle{alphaabbr}
\bibliography{refs}

\def\cprime{$'$} \def\polhk#1{\setbox0=\hbox{#1}{\ooalign{\hidewidth
  \lower1.5ex\hbox{`}\hidewidth\crcr\unhbox0}}}

\end{document}